\documentclass[a4paper,11pt]{article}

\usepackage[T1]{fontenc}
\usepackage[english]{babel}
\usepackage{mathtools}
\usepackage{amsmath,amssymb}
\usepackage{amsmath,amssymb,amsthm,amsfonts}
\usepackage[margin=3.7cm]{geometry}
\usepackage{bbm}
\usepackage{eucal}
\usepackage{graphicx}

\usepackage{tikz,pgfplots}

\usepackage{comment}

\usepackage{hyperref}
  
\usetikzlibrary{plotmarks}
   
\usepackage{fancyhdr}
\newtheorem{theorem}{Theorem}
\newtheorem{lemma}[theorem]{Lemma}

\newtheorem{corollary}[theorem]{Corollary}

\newtheorem{remark}[theorem]{Remark}
\theoremstyle{definition}

\numberwithin{theorem}{section}
\numberwithin{equation}{section}

\newcommand{\ddiv}{\operatorname{div}}
\newcommand{\tri}{\mathcal{T}}
\newcommand{\R}{\mathbb{R}}

\newcommand{\Ker}{\mathrm{Ker}}

\newcommand{\cor}{\mathcal{C}}
\newcommand{\energy}{\mathcal{E}}

\newcommand{\dx}{\,dx}
\newcommand{\eps}{\varepsilon}

\newcommand{\VC}{\tilde{V}_H}
\newcommand{\VH}{V_H}
\newcommand{\Vh}{V_h}

\newcommand{\vC}{\tilde{v}_H}
\newcommand{\uC}{\tilde{u}_H}

\newcommand{\vH}{v_H}
\newcommand{\vh}{v_h}

\begin{document}

\author{Roland Maier\thanks{Institut f\"ur Mathematik,
            Universit\"at Augsburg, 
            Universit\"atsstra{\ss}e 14, D-86159 Augsburg, Germany;
            roland.maier@math.uni-augsburg.de, 
            daniel.peterseim@math.uni-augsburg.de
             }\;
         \and Daniel Peterseim\footnotemark[2]
} 
\title{Explicit Computational Wave Propagation in Micro-Heterogeneous Media\thanks{This is a 
		pre-print of an article published in BIT Numerical Mathematics. The final authenticated version is
		available online at: \url{https://doi.org/10.1007/s10543-018-0735-8}.
		 The authors acknowledge support by Deutsche Forschungsgemeinschaft in the Priority Program 1748 {\it Reliable simulation techniques in solid mechanics} (PE2143/2-2) and thank the Hausdorff Institute for Mathematics in Bonn for the kind hospitality 
 during the trimester program on multiscale problems in 2017.}}
 \date{\today}
 \maketitle
 
\begin{abstract}
\noindent
Explicit time stepping schemes are popular for linear acoustic and elastic wave propagation due to their simple nature which does not require sophisticated solvers for the inversion of the
stiffness matrices. However, explicit schemes are only stable if the time step size is bounded by the mesh size in space subject to the so-called CFL condition. In micro-heterogeneous media, this condition is typically prohibitively restrictive because spatial oscillations of the medium need to be resolved by the discretization in space. This paper presents a way to reduce the spatial complexity in such a setting and, hence, to enable a relaxation of the CFL condition. This is done using the Localized Orthogonal Decomposition method as a tool for numerical homogenization. A complete convergence analysis is presented with appropriate, weak regularity assumptions on the initial data.
\end{abstract}
 
\noindent
{\small\textbf{Keywords} explicit time stepping, hyperbolic equation, heterogeneous media, numerical homogenization, multiscale method
}

\noindent
{\small\textbf{AMS subject classification}
	65M12, 
	65M60, 
	35L05  
}

\section{Introduction}\label{s:intro}

We consider the discretization of the wave equation
\begin{equation}\label{eq:strongform}
\begin{aligned}
\ddot{u} - \ddiv A \nabla u &= f &&\text{in }(0,T)\times\Omega,\\
u(0) &= u_0 \qquad&&\text{in }\Omega,\\
\dot{u}(0)&=v_0 &&\text{in }\Omega,\\
u\vert_{\Gamma} &= 0 &&\text{in }(0,T),\\
\nabla u \cdot \nu\vert_{\partial \Omega \setminus \Gamma} &= 0 &&\text{in }(0,T)
\end{aligned}
\end{equation}
on a polygonal, convex, bounded Lipschitz domain $\Omega\subseteq\R^d$, $d\in\{2,3\}$ with outer normal $\nu$ and Dirichlet boundary $\Gamma \subseteq \partial \Omega$ with non-zero measure. Further, we assume that the initial data $u_0 \in H^1_{\Gamma}(\Omega)$, $v_0 \in L^2(\Omega)$ and the time-independent rough coefficient $A \in L^\infty(\Omega;\R^{d \times d}_{\mathrm{sym}})$ fulfills the bounds $\alpha |\xi|^2 \leq A(x)\xi \cdot \xi$ and $|A(x) \xi| \leq \beta |\xi|$ for all $\xi \in \R^d$ and almost all $x \in \Omega$ for some $0 < \alpha \leq \beta < \infty$. We have in mind coefficients that vary on some small scale $0 <\eps \ll 1$ but we do not need restrictive assumptions such as periodicity or scale separation. 

In order to compute a numerical approximation of problem \eqref{eq:strongform}, we first write the problem in variational form, i.e., we seek a weak solution $u \in L^2(0,T;H^1_{\Gamma}(\Omega))$ with $\dot u \in L^2(0,T;L^2(\Omega))$ and $\ddot u \in L^2(0,T;H^{-1}(\Omega))$ such that
\begin{equation}\label{eq:weak}
\langle \ddot u, v\rangle_{H^{-1}(\Omega) \times H^1(\Omega)} + a(u,v) = (f,v)_{L^2(\Omega)} 
\end{equation}
\noindent for all $v\in H^1_{\Gamma}(\Omega)$ with initial conditions $u(0) = u_0$ and $\dot u(0) = v_0$, where $a$ denotes the bilinear form $a(u,v) \coloneqq \int_\Omega A\nabla u \cdot \nabla v \dx$. 
Note that for any $u_0 \in L^2(\Omega)$, $v_0 \in H^1_{\Gamma}(\Omega)$ and $f \in L^2(0,T;L^2(\Omega))$, there exists a unique weak solution $u$ of \eqref{eq:weak}.
	\noindent A proof of this can, for example, be found in \cite[Ch. 7.2]{evans2010partial}.
Restricting the solution space $H^1_{\Gamma}(\Omega)$ in \eqref{eq:weak} to a finite element space $V_h$ based on a regular mesh $\tri_h$ of $\Omega$ with mesh size $h$ and applying the leapfrog scheme with step size $\Delta t$ in time leads to the following discrete problem:\\

\itshape Find $\mathbf{u_h} = (u_h^n)_{n=0,...N}$ with $u_h^n \in V_h$ such that for $n \geq 2$
\begin{equation}\label{eq:standardfem}
\Delta t^{-2} (u_h^{n+1} - 2 u_h^n + u_h^{n-1}, v_h)_{L^2(\Omega)} + a(u_h^n,v_h) = (f(n \Delta t), v_h)_{L^2(\Omega)} 
\end{equation}
\indent for all $v_h \in V_h$ and given $u_h^0$ and $u_h^1$.\\

\noindent \upshape It is well understood that such a method only leads to acceptable results if the mesh size $h$ is small enough to resolve the fine scale features in space originating from the highly varying coefficient $A$. Consider, for example, a coefficient that oscillates periodically between $\alpha$ and $\beta$ with period length $0 < \eps \ll 1$. In this case, the error of the finite element method scales at best like $(h/\eps)^s$ for some $s > 0$ that depends on the smoothness of $A$ and the domain $\Omega$. In order to obtain accuracy, at least $h <\eps$ should hold. Such an $h$, however, may be too small to allow for reasonably fast computations. It is especially very restrictive since reducing the size of $h$ directly leads to larger systems of linear equations that need to be solved in every time step. Furthermore, the fact that the above method \eqref{eq:standardfem} is explicit in time also introduces the so-called CFL condition that limits the size of the time step $\Delta t$ by the (minimal) mesh size $h_{\text{min}}$, i.e., $\Delta t \lesssim h_{\text{min}}$. It is, hence, too expensive to pose the discrete problem on meshes with small mesh sizes $h$ that resolve fine scale features. The next section introduces a way to cope with the fine scale characteristics on an arbitrarily chosen coarse scale $H$ which not only reduces the size of present linear systems but also allows larger time steps subject to a relaxed CFL condition $\Delta t \lesssim H$.

The approach is based on the so-called Localized Orthogonal Decomposition method (LOD) introduced in \cite{maalqvist2014localization} (see also \cite{P16}) and uses ideas similar to the ones presented in \cite{peterseim2016relaxing} for the wave equation in homogeneous media posed on domains with re-entrant corners. The basic idea of the method is to define low-dimensional finite element spaces that include spatial fine scale features. {The construction is based on the decomposition of the solution space $H^1_{\Gamma}(\Omega)$ into an infinite-dimensional fine scale space and its finite-dimensional $a$-orthogonal complement. The latter has improved approximation properties compared to classical finite elements and may thus be used as both trial and test space for the spatial discretization. It can also be shown that there exists a bijective transformation from the classical finite element space to this improved approximation space. Thus, a basis of the new space is constructed by modifying the classical finite element basis functions by adding the solutions of auxiliary elliptic problems (so-called corrector problems). The corrector problems may even be localized without severely effecting approximation properties, which gives the method its name. The approach has been successfully applied to time-harmonic wave propagation to eliminate the pollution effect \cite{P17,GP15,BGP17}. For the wave equation with rough coefficients, the LOD has already been used in combination with an implicit time discretization (Crank-Nicolson) in \cite{abdulle2017localized}. Therein, the need for additional regularity assumptions on the initial data is discussed, which is also crucial for the explicit time discretization in our case. Another possibility to resolve fine scale features in space is the Heterogeneous Multiscale Method (HMM) \cite{ee2003heterognous,ee2005heterogeneous}, which is for instance discussed in \cite{abdulle2011finite,EHR11} or in \cite{EHR12,AR14} in the context of wave propagation over long time. However, the HMM requires scale separation and may thus not be accurate in the general setting of this work. Another method for the numerical homogenization of the wave equation can be found in \cite{owhzhang08}. There, the idea is to use a harmonic coordinate transformation in order to obtain higher regularity of the weak solution. The main drawbacks of this approach are the necessary assumptions (so-called Cordes-type condition) that are hard to verify, and the approximation of the coordinate transformation for which global fine scale problems need to be solved. Another approach by the same authors is presented in \cite{owhzhang14}, where so-called rough polyharmonic splines based on more demanding biharmonic corrector problems are introduced. A more recent approach \cite{owhzhang17} is based on a decomposition into orthogonal spaces in the spirit of the LOD method and shows the possible generalization of the present approach to a multilevel setting. 

In general, any of the methods mentioned above can be used for the spatial discretization. The advantage of the LOD method is that it preserves the finite element structure of the problem and it is thus very convenient for practical applications. The use of an explicit time stepping scheme on the other side is motivated by its simple nature that allows for faster computations in every time step and by the fact that the discrete energy is conserved (see \eqref{eq:energycons}). Since solutions to the wave equation conserve energy in the continuous setting, such a property is very natural and desirable in the discrete setting as well.

The paper is structured as follows. In Section~\ref{s:numups}, we introduce an idealized method based on the LOD method for the spatial discretization and the leapfrog scheme in time. We show stability and error estimates under suitable regularity assumptions and discuss a simplification of the method. Section~\ref{s:practasp} is devoted to a complete analysis of the fully discrete practical method, where also the auxiliary corrector problems are discretized in order to allow for practical computations. In Section~\ref{s:numres}, numerical experiments are presented to illustrate the numerical performance of the method and we give a short conclusion in Section~\ref{s:concl}.

In the remaining part of this paper we use the notations $(\bullet,\bullet) \coloneqq (\bullet,\bullet)_{L^2(\Omega)}$ and $\|\bullet\| \coloneqq \|\bullet\|_{L^2(\Omega)}$ for the standard $L^2$ scalar product and the corresponding norm. We denote with $V \coloneqq H^1_{\Gamma}(\Omega)$ the space $H^1(\Omega)$ with zero traces on $\Gamma$ and write $a \lesssim b$ if $a \leq C b$ with a generic constant $C$ that can depend on the exact solution $u$ and its higher order time derivatives at time zero as well as the right-hand side $f$, in order to shorten the notation. Further, $\lesssim_T$ indicates linear dependence of the constant $C$ on $T$.

\section{The Idealized Method}\label{s:numups}
As mentioned above, the aim of this section is to discretize problem \eqref{eq:weak} on a coarse mesh with mesh size $H$ that does not resolve fine scale characteristics of the coefficient. The discrete solution should still achieve reasonably good accuracy. The general idea of the LOD is to `correct' coarse finite element functions in such a way that they incorporate fine scale features of the given problem. The following subsection focuses on the spatial discretization and some useful properties. In subsection~\ref{ss:timedisc}, we then introduce an idealized method
and discuss its properties in the remaining subsections of this section.

\subsection{Numerical upscaling by LOD}\label{ss:numups}
We consider a quasi-uniform shape regular mesh $\tri_H$ on $\Omega$ with mesh size $H > \eps$. The corresponding standard $P_1$/$Q_1$ finite element space is denoted by $\VH$. The construction of the modified finite element space is based on a projective quasi-interpolation operator $I_H\colon V\to \VH$ with approximation and stability properties, i.e.,
\begin{align}\label{e:IHapproxstab}
\|H^{-1}(v-I_H v)\| + \|\nabla I_H v\| \leq C_{I_H} \|\nabla v\|
\end{align}
for all $v\in V$ and 
\begin{align}\label{e:IHL2stab}
\|I_H v\| \leq C_{I_H} \|v\|
\end{align}
for all $v\in V$. The constant $C_{I_H}$ only depends on the shape regularity of the elements in the mesh but not on $H$. Ideally, such an operator is also local in the sense that the support of the interpolation is only marginally larger than the support of the original function. This is, for instance, the case with the following possible choice which is used for our numerical experiments, see \cite{oswald1993bpx,brenner1994two,ern2017finite,KPY17}. 

We define $I_H := E_H \circ \Pi_H$, where $\Pi_H$ is the piecewise $L^2$ projection onto $P_1(\tri_H)$/ $Q_1(\tri_H)$, the space of piecewise linear/bilinear and possibly discontinuous functions that vanish at the boundary. $E_H$ denotes the averaging operator that maps $P_1(\tri_H)$/ $Q_1(\tri_H)$ to $\VH$ by assigning to each free vertex the arithmetic mean of the corresponding function values of the neighboring elements. Rigorously, for any $v\in P_1(\tri_H)$/$v\in Q_1(\tri_H)$ and a free vertex $z$ of $\tri_H$, we have 
\begin{equation*}
(E_H(v))(z) = \sum_{T\in\tri_H\atop\text{with } z \in T} v\vert_T(z)\cdot\mathrm{card}^{-1}\{K \in \tri_H : z \in K\}.
\end{equation*}
With such an interpolation operator $I_H$ we can define the so-called \emph{fine scale space} $W$ as its kernel, i.e., $W \coloneqq \Ker I_H$. We can then define for $v \in V$ the corrector $\cor v \in W$ by 
\begin{equation}\label{e:corprob}
a(\cor v,w) = a(v,w)
\end{equation}
for all $w \in W$. Note that for any $v \in V$, it holds that
\begin{align*}
\alpha \|\nabla \cor v\|^2 \leq a(\cor v,\cor v) = a(v,\cor v) \leq \beta \|\nabla v\|\|\nabla\cor v\|,
\end{align*}
and thus
\begin{equation*}\label{eq:stabcor}
\|\nabla\cor v \| \leq \frac{\beta}{\alpha} \|\nabla v\|.
\end{equation*}
Similarly, we also obtain the estimate
\begin{equation*}\label{eq:stabcor2}
\|\nabla\cor v \| \leq \alpha^{-1} C_{I_H} H \|\ddiv A \nabla v\|
\end{equation*}
if $\ddiv A \nabla v \in L^2(\Omega)$ using the approximation property \eqref{e:IHapproxstab}. Define $\VC \coloneqq (1-\cor)\VH$ and observe that $V = \VC \oplus W$ and $a(\VC,W) = 0$ by construction. Further, observe that the inverse inequality holds.
\begin{lemma}[Inverse inequality]
	For any $\vC \in \VC$, 
	\begin{equation}\label{eq:invineq2}
	\|\nabla\vC\| \leq \tilde C_{\mathrm{inv}} H^{-1}\|\vC\|.
	\end{equation}
\end{lemma}
\begin{proof}
	Let $\vC \in \VC$. Since $\vC = (1-\cor)I_H \vC$, we get
	\begin{align*}
	\begin{aligned}
	\alpha \| \nabla \vC\|^2 &\leq a(\vC,\vC) = a(\vC,I_H\vC) \leq \beta \|\nabla \vC\|\|\nabla I_H  \vC\| \\
	&\leq \beta \|\nabla \vC\| C_{\mathrm{inv}} C_{I_H} H^{-1} \|\vC\|
	\end{aligned}
	\end{align*}
	using \eqref{e:IHL2stab} and the standard inverse inequality
	\begin{equation}\label{eq:invineq1}
	\|\nabla\vH\| \leq C_{\mathrm{inv}} H^{-1}\|\vH\|
	\end{equation}
	for any $\vH \in \VH$. Hence, \eqref{eq:invineq2} follows with $\tilde C_\mathrm{inv} \coloneqq C_{\mathrm{inv}} C_{I_H} \beta/\alpha$.
\end{proof}
\noindent Besides, the new space $\VC$ also has the following approximation property, which is a generalization of \cite[Lemma 2.1]{peterseim2016relaxing} to the case of non-constant coefficients.
\begin{lemma}
	For all $u \in V$ with $\ddiv A \nabla u \in L^2(\Omega)$, it holds that
	\begin{align}\label{e:approxVC}
	\inf_{\vC\in \VC} \|u - \vC\|_{H^1(\Omega)} \leq \alpha^{-1}C_{I_H} H\|\ddiv A \nabla u\|
	\end{align}
\end{lemma}
\begin{proof}
	Let $\uC \in \VC$ be the orthogonal projection with respect to the bilinear form $a$ of $u$ onto $\VC$, i.e.,
	\begin{equation*}
	a(\uC,\vC) = a(u,\vC)
	\end{equation*}
	for all $\vC \in \VC$. Therefore, the error $e_H = u - \uC \in W$ and, hence,
	\begin{equation*}
	\alpha \|\nabla e_H\|^2 \leq a(e_H,e_H) = a(u,e_H) = (-\ddiv A \nabla u, e_H) \leq \|\ddiv A \nabla u\|\|e_H\|.
	\end{equation*}
	Since $e_H \in W$, it holds that
	\begin{equation*}
	\|e_H\| = \|(1 - I_H)e_H\| \leq C_{I_H} H \|\nabla e_H\| 
	\end{equation*}
	using the approximation property \eqref{e:IHapproxstab}. Combining both inequalities results in 
	\begin{equation*}
	\|\nabla(u - \uC)\| \leq \alpha^{-1}C_{I_H} H\|\ddiv A \nabla u\|
	\end{equation*}
	which concludes the proof.
\end{proof}

\subsection{Discretization in time}\label{ss:timedisc}
\noindent Based on the adapted spatial discretization defined above and the standard leapfrog scheme in time as in \eqref{eq:standardfem}, the proposed idealized method reads:\\

\itshape Find $\mathbf{\tilde{u}_H} = (\tilde{u}_n)_{n = 0,..,N}$ with $\tilde{u}_n \in \VC$, such that for $n \geq 2$
\begin{equation}\label{eq:lod}
\begin{aligned}
&\Delta t^{-2} (\tilde{u}_{n+1} - 2 \tilde{u}_n + \tilde{u}_{n-1},\vC) + a(\tilde{u}_n,\vC) = (f(n\Delta t),\vC)
\end{aligned}
\end{equation}
\indent for all $\vC \in \VC$ and given $\tilde{u}_0 = (1-\cor)I_H u_0$ and suitable $\tilde{u}_1 \in \VC$.\upshape\\

\noindent We call \eqref{eq:lod} the \emph{idealized method} because we implicitly assume that the corrector problems \eqref{e:corprob} can be computed exactly. In order to show stability and error estimates for this scheme, standard methods can be applied \cite{christiansen2009foundations,joly2003variational}. First, we introduce the discrete energy
\begin{equation*}
\energy^{n+1/2}(\mathbf{\tilde{u}_H}) \coloneqq \frac{1}{2}(\|\dot{\tilde u}_{n+1/2}\|^2 + a(\tilde u_n,\tilde u_{n+1})),
\end{equation*}
where $\dot{\tilde u}_{n+1/2} \coloneqq \frac{\tilde u_{n+1} - \tilde u_n}{\Delta t}$ denotes the discrete time derivative. Using \eqref{eq:lod} with the test function $\vC = \tilde{u}_{n+1} - \tilde{u}_{n-1}$, we derive energy conservation in the sense that
\begin{equation}\label{eq:energycons}
\begin{aligned}
&&&\Delta t \,(f(n\Delta t ),\dot{\tilde{u}}_{n+1/2} + \dot{\tilde{u}}_{n-1/2}) \\
&&=\quad& \Delta t^{-2} (\tilde{u}_{n+1} - 2 \tilde{u}_n + \tilde{u}_{n-1},\tilde{u}_{n+1} - \tilde{u}_{n-1}) + a(\tilde{u}_n, \tilde{u}_{n+1} - \tilde{u}_{n-1})\\
&&=\quad& 2\left(\energy^{n+1/2}(\mathbf{\tilde{u}_H}) - \energy^{n-1/2}(\mathbf{\tilde{u}_H})\right).
\end{aligned}
\end{equation}
\begin{lemma}[Stability of the idealized method]\label{l:stab}
	Assume that the CFL condition
		\begin{equation}\label{eq:cfl}
		1 - \frac{1}{2}\beta C_{\mathrm{inv}}^2C_{I_H}^2 H^{-2} \Delta t^2 \geq \delta
		\end{equation}
		holds for some $\delta > 0$. Then the idealized method \eqref{eq:lod} is stable, i.e.,
	\begin{equation}\label{eq:stability}
	\hspace{-2pt}\| \dot{\tilde{u}}_{n+1/2}\| + \| \nabla \tilde{u}_{n+1}\| \leq C_s \left(\Delta t \sum_{k=1}^n \|f(k \Delta t)\| + \| \dot{\tilde{u}}_{1/2}\| + \| \nabla \tilde{u}_{0}\| + \| \nabla \tilde{u}_{1}\|\right)
	\end{equation}
	with a generic constant $C_s$.
\end{lemma}
\begin{proof}
	The proof mainly follows the ideas presented in \cite{christiansen2009foundations,joly2003variational}, generalized to the case of arbitrary coefficients. With the inverse inequality \eqref{eq:invineq1} and the boundedness of the bilinear form $a$, i.e.,
	\begin{equation*}
	a(u,v) \leq \beta \|\nabla u\|\|\nabla v\| 
	\end{equation*}
	for any $u,v \in V$, we have
	\begin{align*}
	&\quad\,\,\energy^{n+1/2}(\mathbf{\tilde{u}_H})\\ 
	&= \frac{1}{2}\left(\|\dot{\tilde{u}}_{n+1/2}\|^2 + a(\tilde{u}_n,\tilde{u}_{n+1})\right) \\
	&= \frac{1}{4} a(\tilde{u}_{n+1},\tilde{u}_{n+1}) + \frac{1}{4} a(\tilde{u}_{n},\tilde{u}_{n})-\frac{1}{4} a(\tilde{u}_{n+1}-\tilde{u}_{n},\tilde{u}_{n+1}-\tilde{u}_{n}) + \frac{1}{2}\|\dot{\tilde{u}}_{n+1/2}\|^2\\
	&\geq \frac{1}{4} a(\tilde{u}_{n+1},\tilde{u}_{n+1}) + \frac{1}{4} a(\tilde{u}_{n},\tilde{u}_{n}) + \frac{1}{2}(1-\frac{1}{2}\beta C_{\text{inv}}^2C_{I_H}^2H^{-2}\Delta t^2)\|\dot{\tilde{u}}_{n+1/2}\|^2
	\end{align*}
	Here, we have used the fact that, for any $\vH \in \VH$,
	\begin{equation*}
	a((1-\cor)\vH,(1-\cor)\vH) = a(\vH,\vH) - a(\cor\vH,\cor\vH) \leq \beta \|\nabla\vH\|^2.
	\end{equation*}
	\noindent The CFL condition \eqref{eq:cfl} ensures positivity of the discrete energy, i.e.,
	\begin{equation}\label{eq:energybound}
	\energy^{n+1/2}(\mathbf{\tilde{u}_H}) \geq \frac{1}{4} a(\tilde{u}_{n+1},\tilde{u}_{n+1}) + \frac{1}{4} a(\tilde{u}_{n},\tilde{u}_{n}) 
	+ \frac{\delta}{2}\|\dot{\tilde{u}}_{n+1/2}\|^2.
	\end{equation}
	With \eqref{eq:energycons}, we get the estimate
	\begin{align*}
	\energy^{n+1/2}(\mathbf{\tilde{u}_H}) - \energy^{n-1/2}(\mathbf{\tilde{u}_H}) 
	&= \frac{1}{2} \Delta t (f(n \Delta t),\dot{\tilde{u}}_{n+1/2} + \dot{\tilde{u}}_{n-1/2})\\
	&\leq \frac{1}{2} \Delta t \|f(n \Delta t)\| (\|\dot{\tilde{u}}_{n+1/2}\| + \|\dot{\tilde{u}}_{n-1/2}\|)\\
	&\leq \frac{1}{\sqrt{2\delta}} \Delta t \|f(n \Delta t)\| \left(\sqrt{\energy^{n+1/2}(\mathbf{\tilde{u}_H})} + \sqrt{\energy^{n-1/2}(\mathbf{\tilde{u}_H})}\right)
	\end{align*}
	using inequality \eqref{eq:energybound}. From this, we get
	\begin{equation*}
	\sqrt{\energy^{n+1/2}(\mathbf{\tilde{u}_H})} \leq \sqrt{\energy^{n-1/2}(\mathbf{\tilde{u}_H})} + \frac{1}{\sqrt{2\delta}} \Delta t \|f(n \Delta t)\|
	\end{equation*}
	and, hence, the stability estimate
	\begin{equation*}
	\sqrt{\energy^{n+1/2}(\mathbf{\tilde{u}_H})} \leq \sqrt{\energy^{1/2}(\mathbf{\tilde{u}_H})} + \frac{1}{\sqrt{2\delta}} \Delta t \sum_{k=1}^n\|f(k \Delta t)\|.
	\end{equation*}
	This, in turn, implies
	\begin{equation*}
	\| \dot{\tilde{u}}_{n+1/2}\| + \| \nabla \tilde{u}_{n+1}\| \leq C_s \left(\Delta t \sum_{k=1}^n \|f(k \Delta t)\| + \| \dot{\tilde{u}}_{1/2}\| + \| \nabla \tilde{u}_{0}\| + \| \nabla \tilde{u}_{1}\|\right)
	\end{equation*}
	with $C_s = \text{max}\left\{\sqrt{\frac{2}{\delta}},\frac{2}{\sqrt{\alpha}}\right\} \text{max}\left\{\sqrt{\frac{2}{\delta}},\sqrt{\beta}\right\}$.
\end{proof}
\noindent As in \cite{peterseim2016relaxing}, we can use the estimate \eqref{eq:stability} to derive an estimate for the error $\tilde{u}_n-u(t_n)$ in the next subsection. Note that the constant in \eqref{eq:stability} crucially depends on the contrast $\beta/\alpha$ and thus also the constant in the error bound depends on $\beta/\alpha$. 

\subsection{Error analysis}
In order to derive an error estimate, let $z_H\in L^2(0,T;\VC)$ denote the auxiliary semi-discrete solution, i.e., $\dot{z}_H\in L^2(0,T;\VC)$,  $\ddot{z}_H\in L^2(0,T;\VC)$ and $z_H$ solves
\begin{align}\label{eq:semidiscrete}
\left(\ddot{z}_H(t) , \vC\right)_{L^{2}(\Omega)} + a(z_H(t),\vC) = (f(t),\vC)
\end{align}
for all $\vC\in \VC$ and all $t\in[0,T]$, with initial conditions $z_H(0)=(1-\cor)I_H u_0$ and $\dot{z}_H(0)=(1-\cor)I_H v_0$. 
	\begin{remark}
		With the above regularity assumptions $u_0 \in H^1_{\Gamma}(\Omega)$, $v_0 \in L^2(\Omega)$ and $f \in L^2(0,T;L^2(\Omega))$, there exists a unique solution $z_H$ of \eqref{eq:semidiscrete}. The energy norm of $z_H$ can be bounded by the norms of the initial data $u_0$, $v_0$ and the right-hand side $f$. 
		This follows from standard ODE theory due to the fact that \eqref{eq:semidiscrete} may be rewritten as a linear system of ODEs.
	\end{remark}
\noindent For the time discretization, let $N=\lceil T/\Delta t \rceil$ be the number of time steps. Similar to the estimates in \cite{peterseim2016relaxing}, the total error can be split into the discretization error $(E^n)_{n=0,\dots,N}$ in time defined by $E^n=\tilde{u}_n - z_H(n\Delta t)$, and the spatial discretization error $z_H(n\Delta t) - u(n\Delta t) = z_H(n\Delta t)-\Pi_{\VC} u(n\Delta t)-\rho(n\Delta t)$ with the best-approximation error $\rho(t)=u(t)-\Pi_{\VC} u(t)$. Here, $\Pi_{\VC} u(t)$ denotes the orthogonal projection of $u(t)$ onto $\VC$ with respect to the bilinear form $a$. Using \eqref{eq:stability}, we get the following result.
\begin{theorem}[Error of the idealized method]\label{t:error}
	If $\ddot{u}\in L^1(0,T;L^2(\Omega))$ and $\ddot{z}_H\in C(0,T;L^2(\Omega))$, it holds with $t_n=n\Delta t$ that
	\begin{equation}\label{eq:errorbound}
	\begin{aligned}
	&\left\|\frac{(\tilde{u}_{n+1}-u(t_{n+1}))
		-(\tilde{u}_{n}-u(t_{n}))}{\Delta t}\right\|
	+ \|\nabla (\tilde{u}_{n+1}-u(t_{n+1}))\|\\
	&\hspace{-2pt}\leq C_s\left(\|\dot{E}^{1/2}\| + \|\nabla E^1\|
	+ \left\|\dot{z}_H(0)-\Pi_{\VC} \dot{u}(0)\right\|
	+ \|\nabla (z_H(0)-\Pi_{\VC} u(0))\| \right.\\
	&\qquad \qquad  \;\left.
	+ \left\|\frac{\rho(t_n)-\rho(t_{n-1})}{\Delta t}\right\|
	+ \|\nabla \rho(t_n)\|
	+ \int_0^{t_n}
	\left\|\ddot{\rho}(s)\right\| \,ds
	\right.\\
	&\qquad\qquad  \;    \left.   
	+ \sum_{k=1}^n \Delta t
	\left\|\frac{z_H(t_{k+1}) - 2 z_H(t_k)
		+ z_H(t_{k-1})}{(\Delta t)^{2}}
	- \ddot{z}_H(t_k)\right\|\right),\\
	\end{aligned}
	\end{equation}
	with the constant $C_s$ from \eqref{eq:stability}.
\end{theorem}
\noindent With \eqref{e:approxVC} and under the assumption of some additional regularity and appropriate initial conditions, the right-hand side of \eqref{eq:errorbound} scales like $H + \Delta t^2$.
\begin{corollary}[Error of the idealized method]\label{c:error}
	Assume that $u\in C^3(0,T;L^2(\Omega))$, $\ddot u \in C(0,T;H^1(\Omega))$, $z_H\in C^4(0,T;L^2(\Omega))$ and $f\in C^1(0,T;L^2(\Omega))$ and that the corresponding norms are independent of the roughness of $A$, i.e., they do not grow when reducing the fine scale $\eps$ on which the coefficient varies. Let $\tilde u_H(t)$ be the piecewise linear function that interpolates $\mathbf{\uC}$ in time. Then 
	\begin{equation*}\label{eq:totalerror}
	\|u - \tilde{u}_H\|_{L^2(0,T;H^1_{\Gamma}(\Omega))} \lesssim_T H + \Delta t^2.
	\end{equation*}
\end{corollary}
\noindent Note that the standard assumptions on $u$, $z_H$ and $f$ are $u\in C^4(0,T;L^2(\Omega))$, $z_H\in C^4(0,T;L^2(\Omega))$ and $f\in C^2(0,T;L^2(\Omega))$. However, with \eqref{e:IHapproxstab} and the fact that $\rho(t) \in W$ the above assumptions are sufficient to bound the fifth to seventh term in \eqref{eq:errorbound} by $H$.

\subsection{Regularity}\label{ss:reg}
We shall finally show that the regularity conditions of Corollary~\ref{c:error} can be met for relevant classes of problems with arbitrarily rough coefficients that are characterized by the right-hand side $f$ and the initial conditions.
\begin{lemma}[Regularity]\label{l:regts}
	Let $u$ be the solution of \eqref{eq:weak} and $z_H$ the semi-discrete solution of \eqref{eq:semidiscrete}. Suppose that $f \in H^{3}(0,T;L^2(\Omega))$ and
	\begin{itemize}
		\addtolength{\itemindent}{0.3cm}
		\item[$\mathrm{(A1)}$] $v_0 \in H^1_{\Gamma}(\Omega) $,
		\item[$\mathrm{(A2)}$] $\ddot u(0) := f(0) + \ddiv A \nabla u_0 \in H^1_\Gamma(\Omega)$,
		\item[$\mathrm{(A3)}$] $u^{(3)}(0) := \dot f(0) + \ddiv A \nabla v_0 \in H^1_\Gamma(\Omega)$,
		\item[$\mathrm{(A4)}$] $u^{(4)}(0) := \ddot f(0) + \ddiv A \nabla \ddot u(0) \in L^2(\Omega)$. 
	\end{itemize}
	Further assume that the corresponding norms can be bounded independently of the fine scale $\eps$ on which $A$ varies. 
	Then $u$ and $z_H$ satisfy the assumptions of Corollary~\ref{c:error}, i.e., $u\in C^3(0,T;L^2(\Omega))$, $\ddot u \in C(0,T;H^1(\Omega))$ and $z_H\in C^4(0,T;L^2(\Omega))$.
\end{lemma}
\begin{proof}
	Differentiating \eqref{eq:strongform} with respect to time and using the assumptions (A1)-(A4) shows that the time derivatives of $u$ solve wave-type equations as well. From \cite[Ch. 7.2]{evans2010partial} we get the necessary regularity $u \in H^4(0,T;L^2(\Omega))$, from which it follows that $u\in C^3(0,T;L^2(\Omega))$. It further holds that $u \in H^3(0,T;H^1(\Omega))$ and thus $\ddot u \in C(0,T;H^1(\Omega))$. 
	
	\noindent It remains to show that $z_H\in C^4(0,T;L^2(\Omega))$. As above, we can differentiate \eqref{eq:semidiscrete} with respect to time. Further, we use the fact that the initial conditions are defined by $z_H^{(i)}(0)= (1-\cor)I_H u^{(i)}(0)\in\VC$ for $i = 0,\dots,3$. Since solutions to equations of the form \eqref{eq:semidiscrete} are in $C^2(0,T;\VC)$ by standard ODE theory, it follows that $z_H\in C^4(0,T;\VC)$ which concludes the proof.
\end{proof}
\begin{remark}
	The regularity assumptions (A1)-(A4) and $f \in H^{3}(0,T;L^2(\Omega))$ on the initial data and the right-hand side correspond to the conditions in \cite{abdulle2017localized} for the implicit setting and are referred to as \textit{`well-prepared and compatible of order $3$'}.
\end{remark}

\subsection{A simplified method}\label{ss:simmethod}
The regularity properties of the solution $u$ due to the assumptions (A1)-(A4) and $f \in H^{3}(0,T;L^2(\Omega))$ allow for the following simplification of the method defined in \eqref{eq:lod}. First, observe that \eqref{eq:lod} can be written as an equation for standard finite element functions $\mathbf{{u}_H} = ({u}_n)_{n = 0,..,N}$ with $u_n \in \VH$ using
the explicit characterization $\tilde{u}_n = (1-\cor)u_n$, i.e.,
\begin{equation*}\label{eq:lodVH}
\begin{aligned}
&\Delta t^{-2}\left((1-\cor)(u_{n+1} - 2 u_n + u_{n-1}),(1-\cor)\vH\right) + a\left((1-\cor)u_n,(1-\cor)\vH\right)\\ 
=\quad& \left(f(n\Delta t),(1-\cor)\vH\right)
\end{aligned}
\end{equation*}
for all $\vH \in \VH$. A slightly modified method with reduced computational costs seeks $\mathbf{\bar{u}_H} = (\bar{u}_n)_{n = 0,..,N}$ with $\bar{u}_n \in \VH$ such that 
\begin{equation}\label{eq:lodFEM}
\begin{aligned}
\hspace{-1pt}\Delta t^{-2}\left(\bar u_{n+1} - 2 \bar u_n + \bar u_{n-1},\vH\right) + a\left((1-\cor)\bar u_n,(1-\cor)\vH\right) = \left(f(n\Delta t),\vH\right)
\end{aligned}
\end{equation}
for all $\vH \in \VH$. Note that the solution of \eqref{eq:lodFEM} also fulfills stability properties similar to \eqref{eq:stability}. Analogically to \eqref{eq:errorbound}, we can thus also show that
\begin{equation*}\label{eq:totalerror2}
\|u - (1-\cor)\bar{u}_H\|_{L^2(0,T;H^1_{\Gamma}(\Omega))} \lesssim_T H + \Delta t^2
\end{equation*}
under the assumption that the regularity properties of Lemma~\ref{l:regts} hold. Hence, it is reasonable to use the simplified method in practice. See also Chapter~\ref{s:numres}.
\begin{remark}
	The simplification in \eqref{eq:lodFEM} might raise the question whether mass lumping is also a possible modification. Numerical experiments show that mass lumping only works if the coefficient is essentially constant and can have a significant effect on the convergence rate otherwise. This is related to the fact that for general coefficients additional $H^2$ regularity in space cannot be expected.
\end{remark}

\section{The Practical Method}\label{s:practasp}
The method discussed in Section~\ref{s:numups} is idealized in the sense that we have implicitly assumed that the corrector problems \eqref{e:corprob} can be solved exactly. In practice, those problems are discretized and localized as explained in Section~\ref{ss:fsdisc} and \ref{ss:loccor}. 

\subsection{Discretization of fine scales}\label{ss:fsdisc}
As a first step, the problems \eqref{e:corprob} are discretized using classical finite elements.
To quantify the error introduced by such a procedure, let $\mathbf{\tilde{u}_H} = (\tilde{u}_n)_{n = 0,..,N}$ with $\tilde u_n = (1 - \cor)u_n \in \VC$ be the solution of problem \eqref{eq:lod}. Further, define for any $\vH \in \VH$ the discretized correction $\cor_h \vH$ as the finite element solution of \eqref{e:corprob} based on a discrete space $W_h \subset W$ on a mesh $\mathcal{T}_h$ with mesh size $h \leq \eps$, i.e., the mesh size is chosen small enough to resolve variations of the coefficient $A$. Note that $W_h \subset V_h$, with $V_h$ being the standard $P_1$/$Q_1$ finite element space based on the mesh $\mathcal{T}_h$. Denote by $\mathbf{\tilde{u}_{H,h}} = (\tilde{u}_{h,n})_{n = 0,..,N}$ with $\tilde u_{h,n}=(1-\cor_h) u_{h,n}$ the solution of \eqref{eq:lod} in the space $(1-\cor_h)\VH$. The following lemma quantifies the difference of these two solutions.
\begin{lemma}[Fine scale discretization error]\label{l:discerror}
	Under the assumptions of Lemma~\ref{l:stab} and Lemma~\ref{l:regts}, it holds that 
	\begin{equation*}
	\left\|\frac{\tilde u_n - \tilde u_{n-1}}{\Delta t} - \frac{\tilde u_{h,n} - \tilde u_{h,n-1}}{\Delta t}\right\| + \|\nabla(\tilde u_n - \tilde u_{h,n})\| \lesssim_T d_{\VC}[V_h] + H^{-1}(d_{\VC}[V_h])^2,
	\end{equation*}
	with the approximation error $d_{\VC}[V_h]$ defined by
	\begin{equation*}
	d_{\VC}[V_h] \coloneqq \sup_{v \in \VC}\frac{\inf_{\vh \in \Vh}\|\nabla(v - \vh)\|}{\|\nabla v\|}.
	\end{equation*}
\end{lemma}
\begin{proof}
	Observe that the error $\tilde e_n = (1-\cor)(u_n - u_{h,n})$ solves
	\begin{equation*}\label{eq:loddiscerr1}
	\begin{aligned}
	&\Delta t^{-2}\left( \tilde e_{n+1} - 2  \tilde e_n +  \tilde e_{n-1},(1-\cor)\vH\right)+ a\left(\tilde e_n,(1-\cor)\vH\right)\\ 
	=\quad& (-f (n\Delta t),(\cor-\cor_h)\vH) + \Delta t^{-2}(\tilde u_{h,n+1} - 2 \tilde u_{h,n} + \tilde u_{h,n-1},(\cor-\cor_h)\vH) \\
	&+ \Delta t^{-2}((\cor-\cor_h)(u_{h,n+1} - 2 u_{h,n} + u_{h,n-1}),(1-\cor)\vH) \\
	&+a((\cor-\cor_h) u_{h,n},(\cor-\cor_h)\vH) \eqqcolon F^n((1-\cor)\vH)
	\end{aligned}
	\end{equation*}
	for all $\vH \in \VH$. If ${F^n\vert}_{\VC} \in L^2(\Omega)$, we can derive a bound on the error using similar arguments as in the derivation of \eqref{eq:stability}. First, we need to estimate $\|\nabla(\cor-\cor_h)\vH\|$. Using the fact that the corrector problems \eqref{e:corprob} can be interpreted as saddle point problems, we obtain from finite element saddle point theory \cite[Ch. 2.2.2]{brezzi2012mixed} that
	\begin{equation}\label{e:diffcor1}
	\begin{aligned}
	\|\nabla(\cor-\cor_h)\vH\| &\leq (1 + C_{I_H})(1 + \beta/\alpha) \inf_{\vh \in \Vh} \|\nabla(\cor \vH - \vh)\| \\
	&\leq C_* \sup_{v \in \VC}\frac{\inf_{\vh \in \Vh}\|\nabla(v - \vh)\|}{\|\nabla v\|}\|\nabla \vC\| \\
	& = C_*\, d_{\VC}[V_h]\, \|\nabla\vC\|
	\end{aligned}
	\end{equation}
	with $C_* = (1 + C_{I_H})(1 + \beta/\alpha)$ and $\vC = (1-\cor)\vH$ or $\vC = (1-\cor_h)\vH$. Using \eqref{e:IHapproxstab}, we further get
	\begin{equation}\label{e:diffcor2}
	\|(\cor-\cor_h)\vH\| \leq C_{I_H}H\,C_* \, d_{\VC}[V_h]\,\|\nabla\vC\|.
	\end{equation}
	With \eqref{e:diffcor1}, \eqref{e:diffcor2} and the inverse inequality \eqref{eq:invineq2}, we can derive the following bound,
	\begin{equation*}
	\begin{aligned}
	\sup_{\vC \in \VC}\frac{|F^n(\vC)|}{\|\vC\|} &\leq \quad\Bigg( C_{I_H}C_*\tilde C_\mathrm{inv} \Big(\|f(n\Delta t)\| + 2\left\|\frac{\tilde u_{h,n+1} - 2\tilde u_{h,n} + \tilde u_{h,n-1}}{\Delta t^2}\right\|\Big) \\
	&\quad + \,\,\beta C_*^2\tilde C_\mathrm{inv}H^{-1}\|\nabla \tilde u_{h,n}\|\,d_{\VC}[V_h] \Bigg) d_{\VC}[V_h]\\
	&\lesssim \quad\, d_{\VC}[V_h] + H^{-1}(d_{\VC}[V_h])^2.
	\end{aligned}
	\end{equation*}
	Thus, using the above equations and the fact that, for any $v_H,y_H \in \VH$, 
	\begin{equation*}
	\|(1-\cor)v_H - (1-\cor_h)y_H\|_\bullet \leq \|(1-\cor)(v_H-y_H)\|_\bullet + \|(\cor-\cor_h)y_H\|_\bullet,
	\end{equation*}
	it follows as in the derivation of \eqref{eq:stability} that
	\begin{equation*}\label{e:totaldiscerror1}
	\left\|\frac{\tilde u_n - \tilde u_{n-1}}{\Delta t} - \frac{\tilde u_{h,n} - \tilde u_{h,n-1}}{\Delta t}\right\| + \|\nabla(\tilde u_n - \tilde u_{h,n})\| \lesssim_T d_{\VC}[V_h] + H^{-1}(d_{\VC}[V_h])^2.
	\end{equation*}
\end{proof}
\noindent Note that we have used the fact that $\Delta t^{-2}(\tilde u_{h,n+1} -2\tilde u_{h,n} +\tilde u_{h,n-1})$ can be bounded in $L^2$ independently of $\Delta t$ and $H$. To see this, let $\dot{\tilde{u}}_{h,n+1/2} = \Delta t^{-1}(\tilde u_{h,n+1} -\tilde u_{h,n})$ be the discrete time derivative and observe that $(\dot{\tilde{u}}_{h,n+1/2})_{n=0,...,N-1}$ solves
\begin{equation*}
\begin{aligned}
&\Delta t^{-2}\left(\dot{\tilde{u}}_{h,n+3/2} - 2 \dot{\tilde{u}}_{h,n+1/2} + \dot{\tilde{u}}_{h,n-1/2}),\vC\right) + a\left(\dot{\tilde{u}}_{h,n+1/2},\vC\right)\\ 
=\quad& \left(\Delta t^{-1}(f((n+1)\Delta t)-f(n\Delta t)),\vC\right)
\end{aligned}
\end{equation*}
for all $\vC \in (1-\cor_h)\VH$. Therefore, with equation~\eqref{eq:stability} and the regularity assumptions of Lemma~\ref{l:regts} we can bound the $L^2$ norm of $\Delta t^{-2}(\tilde u_{h,n+1} -2\tilde u_{h,n} +\tilde u_{h,n-1})$ in terms of the initial data and the right-hand side.

\subsection{Localization of correctors}\label{ss:loccor}
As a next step, we want to define the solution $\mathbf{\tilde{u}^\ell_{H,h}}$ of the fully practical method and quantify the error between the solutions $\mathbf{\tilde{u}_{H,h}}$ and $\mathbf{\tilde{u}^\ell_{H,h}}$. The \emph{practical method} reads:\\

\itshape Find $\mathbf{\tilde{u}^\ell_{H,h}} = (\tilde{u}^\ell_{h,n})_{n = 0,..,N}$ with $\tilde u^\ell_{h,n}=(1-\cor^\ell_h) u^\ell_{h,n} \in (1-\cor^\ell_h)\VH$, such that \\\indent for $n \geq 2$
\begin{equation}\label{eq:practmethod}
\begin{aligned}
&\Delta t^{-2} (\tilde u^\ell_{h,n+1} - 2 \tilde u^\ell_{h,n} + \tilde u^\ell_{h,n-1},\vC) + a(\tilde u^\ell_{h,n},\vC) = (f(n\Delta t),\vC)
\end{aligned}
\end{equation}
\indent for all $\vC \in (1-\cor^\ell_h)\VH$ and given $\tilde{u}^\ell_{h,0} = (1-\cor^\ell_h)I_H u_0$ and suitable $\tilde{u}^\ell_{h,1} \in$ \\\indent$(1-\cor^\ell_h)\VH$.\upshape\\

\noindent Here, for any $\vH \in \VH$, $\cor^\ell_h \vH$ denotes the discretized solution of \eqref{e:corprob} which is restricted to computations on local patches with $\ell$ layers. 
To be more precise, we first rewrite the operator $\cor_h\colon V_H \to W_h$ as
	\begin{equation*}
	\cor_h v_H = \cor_h\left(\sum_{T\in \tri_H} v_H\vert_T\right) = \cor_h\left(\sum_{T\in \tri_H}\sum_{i = 1}^{n_T} v_H(x_{T,i})\Lambda_{T,i}\right) = \sum_{T\in \tri_H}\sum_{i = 1}^{n_T} v_H(x_{T,i}) q_{T,i},
	\end{equation*}
	where $x_{T,i},\ i=1,\dots n_T$ denote the corner points of $T$ and $\Lambda_{T,i}$ the corresponding nodal basis functions on $T$. Further, the element correctors $q_{T,i}\in W_h$ are defined by
	\begin{equation*}
	a(q_{T,i}, w_h) = a\vert_T(\Lambda_{T,i}, w_h)
	\end{equation*}
	for all $w_h \in W_h$, where $a\vert_T(u,v) \coloneqq \int_T A\nabla u \cdot \nabla v \dx$. Similarly, we can define the operator $\cor^\ell_h\colon V_H \to W_h$ by
	\begin{equation*}
	\cor^\ell_h v_H = \sum_{T\in \tri_H}\sum_{i = 1}^{n_T} v_H(x_{T,i}) q^\ell_{T,i},
	\end{equation*}
	where the localized element correctors $q^\ell_{T,i} \in W_h(\mathcal{N}^\ell(T))$ are given by
	\begin{equation*}
	a(q^\ell_{T,i}, w_h) = a\vert_T(\Lambda_{T,i}, w_h)
	\end{equation*}
	for all $w_h \in W_h(\mathcal{N}^\ell(T))$. Here, $\mathcal{N}^\ell(T)$ is the extension of $T$ by $\ell$ layers of elements and $W_h(\mathcal{N}^\ell(T))$ the restriction of $W_h$ to functions with support in $\mathcal{N}^\ell(T)$. See also \cite{maalqvist2014localization} for further details. Alternatively, the operator $\cor^\ell_h$ could be defined as the $\ell$th iterate of some preconditioned solver based on an overlapping domain decomposition, as shown in \cite{KPY17}.

The computation of the correctors is done during the offline stage and can be parallelized. Furthermore, periodic structure may be exploited. The additional cost to solve the corrector problems is moderate and the main advantage of the method lies in the online stage, where smaller linear systems need to be solved and relatively coarse time steps (subject to the CFL condition) may be used.
\begin{lemma}[Localization error]\label{l:locerror}
	Suppose that the assumptions of Lemma~\ref{l:stab} and Lemma~\ref{l:regts} hold. Then
	\begin{equation*}
	\left\|\frac{\tilde u_{h,n} - \tilde u_{h,n-1}}{\Delta t} - \frac{\tilde u^\ell_{h,n} - \tilde u^\ell_{h,n-1}}{\Delta t}\right\| + \|\nabla(\tilde u_{h,n} - \tilde u^\ell_{h,n})\| \lesssim_T \ell^{d/2}e^{-c\ell} + H^{-1}\ell^{d}e^{-2c\ell}.
	\end{equation*}
\end{lemma}
\begin{proof}
	Similar to the findings in Lemma~\ref{l:discerror}, the error $\tilde e_{n}^\ell = (1-\cor_h)(u_{h,n} - u^\ell_{h,n})$ solves 
	\begin{equation*}\label{eq:loddiscerr2} 
	\begin{aligned}
	&\Delta t^{-2}\left( \tilde e^\ell_{n+1} - 2  \tilde e^\ell_n +  \tilde e^\ell_{n-1},(1-\cor_h)\vH\right)+ a\left(\tilde e^\ell_n,(1-\cor_h)\vH\right)\\
	=\quad& (-f (n\Delta t),(\cor_h-\cor^\ell_h)\vH) + \Delta t^{-2}(\tilde u^\ell_{h,n+1} - 2 \tilde u^\ell_{h,n} + \tilde u^\ell_{h,n-1},(\cor_h-\cor^\ell_h)\vH) \\
	&+ \Delta t^{-2}((\cor_h-\cor^\ell_h)(u^\ell_{h,n+1} - 2 u^\ell_{h,n} + u^\ell_{h,n-1}),(1-\cor_h)\vH) \\
	&+a((\cor_h-\cor^\ell_h) u^\ell_{h,n},(\cor_h-\cor^\ell_h)\vH) \eqqcolon F_h^n((1-\cor_h)\vH)
	\end{aligned}
	\end{equation*}
	for all $\vH \in \VH$. As above, we want to show that ${F_h^n\vert}_{(1-\cor_h)\VH} \in L^2(\Omega)$. In \cite{HenP13,maalqvist2014localization}, it is shown that for any $\vH \in \VH$
	\begin{equation*}\label{e:diffcor3}
	\|\nabla(\cor_h-\cor^\ell_h)\vH\| \leq C \ell^{d/2} e^{-c\ell} \|\nabla \vH\|
	\end{equation*}
	and thus also
	\begin{equation*}\label{e:diffcor4}
	\|(\cor_h-\cor^\ell_h)\vH\| \leq C \ell^{d/2} e^{-c\ell} C_{I_H} H \|\nabla v_H\|,
	\end{equation*}
	see also \cite{KY16,KPY17} for an alternative constructive proof. Similar to the estimates in Section~\ref{ss:fsdisc}, we obtain
	\begin{equation*}
	\begin{aligned}
	\sup_{\vC \in (1-\cor_h)\VH}\frac{|F_h^n(\vC)|}{\|\vC\|} &\leq \quad\Bigg( CC^2_{I_H}C_\mathrm{inv}\Big(\|f(n\Delta t)\| + 2\left\|\frac{\tilde u^\ell_{h,n+1} - 2\tilde u^\ell_{h,n} + \tilde u^\ell_{h,n-1}}{\Delta t^2}\right\|\Big) \\
	&\quad + \,\,\beta C^2C_{I_H}C_\mathrm{inv}H^{-1}\|\nabla \tilde u^\ell_{h,n}\|\,\ell^{d/2}e^{-c\ell} \Bigg) \ell^{d/2}e^{-c\ell}\\
	&\lesssim \quad\, \ell^{d/2}e^{-c\ell} + H^{-1}\ell^{d}e^{-2c\ell}
	\end{aligned}
	\end{equation*}
	and finally
	\begin{equation*}\label{e:totaldiscerror2}
	\left\|\frac{\tilde u_{h,n} - \tilde u_{h,n-1}}{\Delta t} - \frac{\tilde u^\ell_{h,n} - \tilde u^\ell_{h,n-1}}{\Delta t}\right\| + \|\nabla(\tilde u_{h,n} - \tilde u^\ell_{h,n})\| \lesssim_T \ell^{d/2}e^{-c\ell} + H^{-1}\ell^{d}e^{-2c\ell}.
	\end{equation*}
\end{proof}

\subsection{Error of the practical method}\label{ss:errfullydisc}
We can now formulate the following theorem using Lemma~\ref{l:discerror}, Lemma~\ref{l:locerror} and the triangle inequality.
\begin{theorem}[Error of the practical method]\label{t:practerror}
	Let $u$ be the solution of \eqref{eq:weak} and assume that $\mathrm{(A1)}$-$\mathrm{(A4)}$ and $f \in H^{3}(0,T;L^2(\Omega))$ hold and $\Delta t \lesssim H$ subject to the CFL condition \eqref{eq:cfl}. Further, let $\mathbf{\tilde{u}^\ell_{H,h}}$ be the solution of the practical method and $\tilde u^\ell_{H,h}(t)$ the piecewise linear function that interpolates $\mathbf{\tilde{u}^\ell_{H,h}}$ in time. Then 
	\begin{equation}\label{eq:errfull1}
	\begin{aligned}
	&\|u - \tilde u^\ell_{H,h}\|_{L^2(0,T;H^1_{\Gamma}(\Omega))} \\\lesssim_T\,\,& H + \Delta t^2 + d_{\VC}[V_h] + H^{-1}(d_{\VC}[V_h])^2 + \ell^{d/2}e^{-c\ell} + H^{-1}\ell^{d}e^{-2c\ell}.
	\end{aligned}
	\end{equation}
	If $\ell \gtrsim |\log H|$, \eqref{eq:errfull1} simplifies to
	\begin{equation*}\label{eq:errfull2}
	\|u - \tilde u^\ell_{H,h}\|_{L^2(0,T;H^1_{\Gamma}(\Omega))} \lesssim_T H + \Delta t^2 + d_{\VC}[V_h] + H^{-1}(d_{\VC}[V_h])^2.
	\end{equation*}
\end{theorem}
\noindent Theorem~\ref{t:practerror} shows that in order to obtain a reasonable error of order $H$, the error introduced by the discretization of the corrector problems \eqref{e:corprob} and thus
the approximation error $d_{\VC}[V_h]$ need to be of order $H$ as well. The following lemma quantifies the approximation error $d_{\VC}[V_h]$ under additional regularity assumptions on the coefficient $A$.
\begin{lemma}\label{l:approxerr}
	Suppose that $A \in W^{1,\infty}(\Omega;\R)$ with oscillations on the scale $\eps$. Further, let $I_h\colon V \to V_h$ be a quasi-interpolation operator that fulfills properties similar to \eqref{e:IHapproxstab} and \eqref{e:IHL2stab} with $h$ instead of $H$. Then, 
	\begin{equation*}
	d_{\VC}[V_h] \lesssim h(H^{-1} + \eps^{-1}).
	\end{equation*}
\end{lemma}
\begin{proof}
	For any $\vC \in \VC$, we have
	\begin{equation*}
	\begin{aligned}
	\inf_{\vh \in \Vh}\|\nabla(\vC - \vh)\| &\leq \|\nabla(1 - I_h)\vC\| \leq C_{I_h} h \|D^2 \vC\| \\&\leq C_{I_h} h \|\Delta \vC\| \leq \frac{h}{\alpha} C_{I_h} \|A \Delta \vC\| \\
	&\leq \frac{h}{\alpha} C_{I_h}\left( \|\ddiv A \nabla \vC\| + \|A\|_{W^{1,\infty}(\Omega;\R)}\|\nabla \vC\|\right)\\
	&\leq \frac{h}{\alpha} C_{I_h}\left(\beta C_{\text{inv}}C_{I_H}H^{-1} + c\eps^{-1}\right) \|\nabla\vC\|,
	\end{aligned}
	\end{equation*}
	using the fact that $\|A\|_{W^{1,\infty}(\Omega;\R)} \lesssim c\eps^{-1}$ and 
	\begin{equation}\label{eq:divA}
	\|\ddiv A \nabla \vC\| \leq \beta C_{\text{inv}}C_{I_H}H^{-1}\|\nabla \vC\|. 
	\end{equation}
	To show this, let $v \in C_c^\infty(\Omega)$ and observe that
		\begin{equation*}
		\begin{aligned}
		\frac{|(\ddiv A \nabla \vC,v)|}{\|v\|} &= \frac{|a(\vC,v)|}{\|v\|} = \frac{|a(\vC,I_H v)|}{\|v\|} \leq \frac{\beta \|\nabla \vC\|\|\nabla I_H v\|}{\|v\|} \\&\leq  \beta C_\mathrm{inv} C_{I_H} H^{-1} \|\nabla \vC\|
		\end{aligned}
		\end{equation*}
		employing the estimates \eqref{e:IHL2stab} and \eqref{eq:invineq1}. The inequality \eqref{eq:divA} then follows by the density of $C_c^\infty$ in $L^2$. Therefore, $d_{\VC}[V_h]$ can be bounded by
	\begin{equation*}
	d_{\VC}[V_h] \lesssim h(H^{-1} + \eps^{-1}).
	\end{equation*}
\end{proof}
\noindent Using Theorem~\ref{t:practerror} and Lemma~\ref{l:approxerr}, we obtain the following result. 
\begin{corollary}[Error of the practical method]\label{c:practerror}
	Assume that $\mathrm{(A1)}$-$\mathrm{(A4)}$ and $f \in H^{3}(0,T;L^2(\Omega))$ hold. Assume further that $A \in W^{1,\infty}(\Omega;\R)$, $\Delta t \lesssim H$ subject to the CFL condition \eqref{eq:cfl}, $\ell \gtrsim  |\log H|$ and $h \lesssim H\eps$. Then 
	\begin{equation*}\label{eq:errfull3}
	\|u - \tilde u^\ell_{H,h}\|_{L^2(0,T;H^1_{\Gamma}(\Omega))} \lesssim_T H + \Delta t^2.
	\end{equation*}
\end{corollary}
\noindent While orders of convergence in space and time appear imbalanced when the error is measured in $L^2(0,T;H^1_{\Gamma}(\Omega))$, quadratic convergence is empirically observed for the  $L^2(0,T;L^2(\Omega))$ norm. In this sense, the error estimates of our explicit method are competitive with the fully implicit Crank-Nicolson approach of \cite{abdulle2017localized} provided that the fine scale discretization errors of \cite{abdulle2017localized} can be bounded by $(h/\eps)^2$.
\begin{remark}
	The assumptions on the fine mesh size $h$ in Corollary~\ref{c:practerror} are in line with the theoretical findings for the well studied elliptic case. Also note that the above construction is not limited to approximation spaces based on $P_1$/$Q_1$ finite elements. 
		In principle, there is no restriction to devise a higher-order variant of the method in space and to combine it with any time stepping approach. 
		However, higher order convergence rates with respect to $H$ can only be achieved if the interpolation operator $I_H$ fulfills additional orthogonality properties and the coefficient $A$ is regular enough. Further, it is important to adjust the number of element layers $\ell$ for the localization accordingly.
\end{remark}

\section{Numerical Results}\label{s:numres}
In this section, we want to present two numerical experiments to illustrate the theoretical results from the previous sections. The computations are done using an adaption of the code from \cite{Hel17}. The error of the method is measured in the discrete $L^2(0,T;H^1_{\Gamma}(\Omega))$ norm
\begin{equation*}
\| v \|^2_{\Omega,N} := \sum_{i = 1}^N \Delta t \,\|v(i \Delta t)\|^2_{H^1_{\Gamma}(\Omega)}
\end{equation*}
where $N = \lceil T/\Delta t\rceil$ denotes the number of time steps. In both numerical examples, the domain is set to $\Omega = (0,1)^2$ and the final time is chosen as $T = 1$. The reference solution is computed using standard finite elements paired with a leapfrog scheme in time on a uniform quadrilateral mesh $\mathcal{T}_h$ with mesh size $h = \sqrt{2}\cdot 2^{-8}$ which is also the mesh parameter for the computations of the corrector problems.
The fine time step size is chosen small enough subject to the standard CFL condition, i.e., $\Delta t_\mathrm{fine} \leq C_\mathrm{CFL} h$, where $C_\mathrm{CFL} = \sqrt{2}\beta^{-1/2} C_\mathrm{inv}^{-1}$. This condition can be shown similarly to \eqref{eq:cfl} and is slightly relaxed compared to \eqref{eq:cfl} since $C_{I_H} \geq 1$ in general. Practical experiments showed that $C_\mathrm{CFL} = \sqrt{2}\beta^{-1/2}0.14$ is a sufficient and rather sharp choice for the stability of both the standard finite element solution and the coarse solution computed by the method stated above on a quadrilateral mesh $\tri_H$ with mesh size $H$. In the following experiments, we set $\Delta t = C_\mathrm{CFL} H$. Note that given $u^\ell_{h,0}$ and approximations $v^\ell_{h,0}$ of $v_0$ and $f^\ell_{h,0}$ of $f(0)$, $u^\ell_{h,1}$ is computed using the second-order Taylor polynomial, i.e.,
\begin{equation*}
(u^\ell_{h,1},\vC) = (u^\ell_{h,0},\vC) + \Delta t \,(v^\ell_{h,0},\vC) - \frac{1}{2} \Delta t^2 \,a(u^\ell_{h,0},\vC) + \frac{1}{2}\Delta t^2\, (f^\ell_{h,0},\vC)
\end{equation*} 
for any $\vC \in (1-\cor^\ell_h)\VH$. This choice is crucial in order to get the optimal convergence rate.

\subsection{Example 1}\label{ss:ex1}
For the first example, we take the setting from \cite[Sec. 6.2]{abdulle2017localized}, i.e., $f = 1$, $u_0 = v_0 = 0$, $\Gamma =\partial \Omega$ and $A$ as depicted in Figure~\ref{fig:ex1} (left), with $\alpha = 0.04$, $\beta = 1.96$ and $\eps = 0.006$. A detailed  formula for the coefficient can be found in \cite[Sec. 6.2]{abdulle2017localized}. 
The parameter $\ell$ is chosen as $\ell = 2$ for all values of $H$. The remaining discretization parameters are defined above. The errors of the practical method are shown in Figure~\ref{fig:ex1} (right). The red curve shows the errors of the standard method defined in \eqref{eq:practmethod} and the blue curve displays the errors of the method based on \eqref{eq:lodFEM} which uses the classical finite element mass matrix. Both curves show the expected linear convergence and are very close which seems to justify the theoretical observation that the mass matrices may be exchanged. 
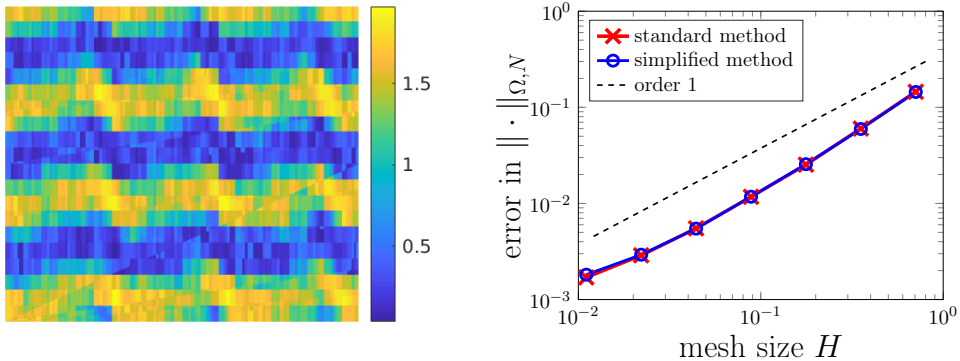
\begin{figure}[h]
	\begin{center}
		\scalebox{.98}{
			\includegraphics[width=0.51\textwidth]{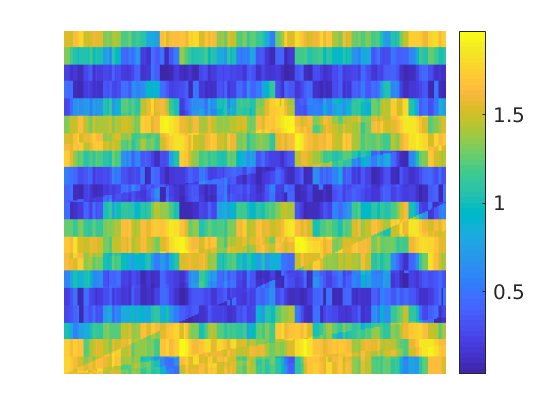}
		}
		\scalebox{0.7}{
			\begin{tikzpicture}
			
			\begin{axis}[%
			width=2.7in,
			height=2.15in,
			at={(0.952in,0.761in)}, 
			scale only axis,
			xmode=log,
			xmin=0.01,
			xmax=1,
			xminorticks=true,
			xlabel={\LARGE mesh size $H$},
			ymode=log,
			ymin=0.001,
			ymax=1,
			yminorticks=true,
			ylabel={\LARGE error in $\|\cdot\|_{\Omega,N}$},
			axis background/.style={fill=white},
			legend style={legend cell align=left, align=left, draw=white!15!black},
			legend pos=north west
			]
			\addplot [color=red, line width=1.9pt, mark size=5.5pt, mark=x]
			table[row sep=crcr]{%
				0.707106781186548	0.14512640958\\
				0.353553390593274	0.0601048660277\\
				0.176776695296637	0.0253425385818\\
				0.0883883476483184	0.0117213930874\\
				0.0441941738241592	0.00550450558784\\
				0.0220970869120796	0.00288351097216\\
				0.0110485434560398	0.0016988927048\\
			};
			\addlegendentry{standard method}
			
			\addplot [color=blue, line width=1.5pt, mark size=3.0pt, mark=o]
			table[row sep=crcr]{%
				0.707106781186548	0.144498959523\\
				0.353553390593274	0.0595028945195\\
				0.176776695296637	0.0256220133354\\
				0.0883883476483184	0.0117408374925\\
				0.0441941738241592	0.00547149515418\\
				0.0220970869120796	0.00292847216739\\
				0.0110485434560398	0.00181426137372\\
			};
			\addlegendentry{simplified method}
			
			\addplot [color=black, line width=0.9pt, dashed]
			table[row sep=crcr]{%
				0.8	0.3\\
				0.012	0.0045\\
			};
			\addlegendentry{order 1}
			
			\end{axis}
			\end{tikzpicture}%
		}
	\end{center}
	\caption{Coefficient $A$ (left) and errors (right) for example 1.}
	\label{fig:ex1}
\end{figure}

\subsection{Example 2}\label{ss:ex2}
In the second example, we choose ${\Gamma = \{ x \in \partial \Omega\colon x_1 = 0\}}$, $f(x,t) = \sin(4\pi x_1)(1-t)$ and $v_0 = 0$. We further let $u_0 \in H^1_\Gamma$ be the solution of
\begin{equation*}
a(u_0,v) = \int_\Omega 5\sin(\pi x_1)\sin(\pi x_2) v \dx
\end{equation*}
for all $v \in H^1_\Gamma$. The coefficient $A$ is shown in Figure~\ref{fig:ex2} (left), where $\alpha = 1$, $\beta = 17.78$ with $\eps = 0.02$. The precise formula for $A$ is
	\begin{equation*}
	\begin{aligned}
	A(x) = 1.9&\cdot\left(\lfloor 2x_1\rfloor\lfloor 8 (1-x_1)\rfloor + \lfloor 2(1-x_1)\rfloor\lfloor 8x_1\rfloor\right)\\&\cdot\left(\lfloor 2x_2\rfloor\lfloor 8 (1-x_2)\rfloor + \lfloor 2(1-x_2)\rfloor\lfloor 8x_2\rfloor\right)\\&\cdot\sin(\lfloor 32x_1 \rfloor)^2 \, \sin(\lfloor 64x_2\rfloor)^2+1.
	\end{aligned}
	\end{equation*}
	The other discretization parameters are chosen as defined above. The red curve in Figure~\ref{fig:ex2} (right) shows the errors of the standard method \eqref{eq:practmethod} with $\ell = 2$ and the green curve shows the errors of the standard method with $\ell = 4$. It can be seen that the red curve stagnates for smaller values of $H$ which is in accordance with the theoretical observations that $\ell$ needs to be  chosen proportional to $|\log(H)|$. The convergence rate is again in line with the theoretical results and seems to be even slightly better for $\ell = 4$ at around $1.5$. Note that, as in the first example, the errors of the simplified method based on \eqref{eq:lodFEM} are very close to the errors of the standard method but are not depicted for better visibility. Also, since the value $\beta$ is only taken in a small part of the domain, the CFL condition can be slightly relaxed for this example. 
\begin{figure}[h]
	\begin{center}
		\scalebox{.98}{
			\includegraphics[width=0.51\textwidth]{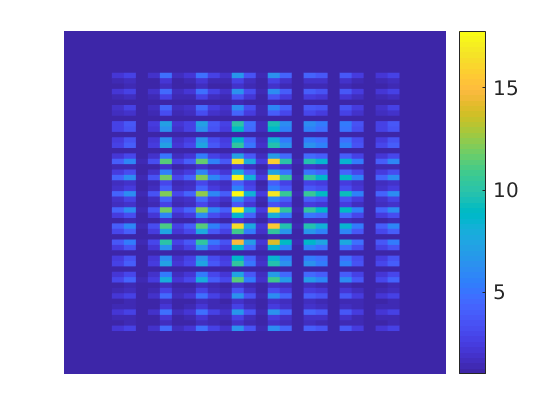}
		}
		\scalebox{0.7}{
			\begin{tikzpicture}
			
			\begin{axis}[%
			width=2.7in,
			height=2.15in,
			at={(0.952in,0.761in)},
			scale only axis,
			xmode=log,
			xmin=0.01,
			xmax=1,
			xminorticks=true,
			xlabel={\LARGE mesh size $H$},
			ymode=log,
			ymin=0.0001,
			ymax=1,
			yminorticks=true,
			ylabel={\LARGE error in $\|\cdot\|_{\Omega,N}$},
			axis background/.style={fill=white},
			legend style={legend cell align=left, align=left, draw=white!15!black},
			legend pos= north west
			]
			\addplot [color=red, line width=1.5pt, mark size=5.5pt, mark=x]
			table[row sep=crcr]{%
				0.707106781186548	0.115929465414\\
				0.353553390593274	0.0692045390136\\
				0.176776695296637	0.0198358436703\\
				0.0883883476483184	0.00585029351279\\
				0.0441941738241592	0.00261857123181\\
				0.0220970869120796	0.0019304494013\\
				0.0110485434560398	0.00100341294064\\
			};
			\addlegendentry{standard method, $\ell=2$}
			
			\addplot [color=green, line width=1.5pt, mark size=4.0pt, mark=diamond]
			table[row sep=crcr]{%
				0.707106781186548	0.115929465414\\
				0.353553390593274	0.0687406999219\\
				0.176776695296637	0.0198420421635\\
				0.0883883476483184	0.00496611083453\\
				0.0441941738241592	0.00138056267706\\
				0.0220970869120796	0.000423052834832\\
				0.0110485434560398	0.000115699586959\\
			};
			\addlegendentry{standard method, $\ell=4$}
			
			\addplot [color=black, line width=0.9pt, dashed]
			table[row sep=crcr]{%
				0.8	0.05\\
				0.012	0.00075\\
			};
			\addlegendentry{order 1}
			
			\end{axis}
			\end{tikzpicture}%
		}
	\end{center}
	\caption{Coefficient $A$ (left) and errors (right) for example 2.}
	\label{fig:ex2}
\end{figure}
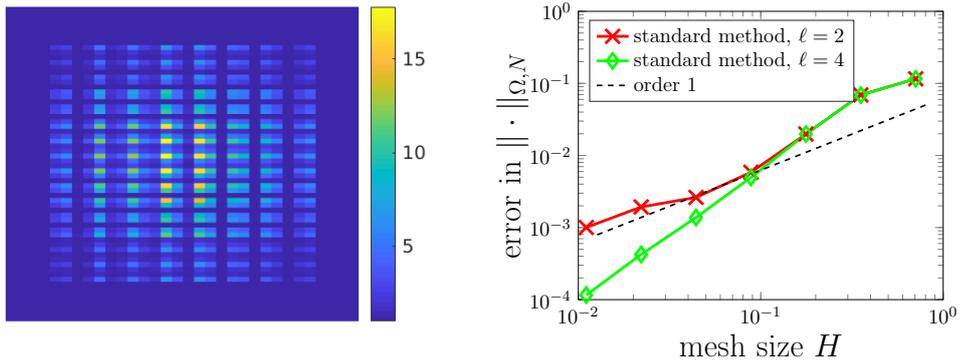

\section{Conclusions}\label{s:concl}
In this work, we have discussed a discretization of the wave equation with rough coefficients. We have used the LOD method in space and the explicit leapfrog scheme for the time discretization and are able to obtain first-order convergence under suitable assumptions on the initial data and subject to a relaxed version of the CFL condition. Numerical experiments illustrate the 
theoretical findings.

Ongoing research aims at further weakening the presented assumptions on the initial data especially in the context of $L^2(L^2)$ error estimates and the generalization to elastic and poroelastic waves based on preparatory work \cite{BG16,ACMPP18}. Additionally, the long-time behavior of numerical solutions to the wave equation will be considered.

\bibliographystyle{plain}

\end{document}